\documentclass{article}
\usepackage{times, amsmath, amsthm, amssymb, graphicx,mathrsfs, amsfonts}
\input xy
\xyoption{all}
\theoremstyle{definition}
\newtheorem{thm}{Theorem}[subsection]
\newtheorem{cor}[thm]{Corollary}
\newtheorem{lem}[thm]{Lemma}
\newtheorem{pro}[thm]{Proposition}

\newtheorem{con}[thm]{Conjecture}

\newtheorem{dfc}[thm]{Definition-Construction}

\newtheorem{df}[thm]{Definition}
\newtheorem{rem}[thm]{Remark}
\newtheorem{exa}[thm]{Example}

\newcommand{\mmA}{\mathscr{A}}
\newcommand{\mmB}{\mathscr{B}}

\newcommand{\mmM}{\mathscr{M}}

\newcommand{\mB}{\mathcal{B}}
\newcommand{\mD}{\mathcal{D}}
\newcommand{\mC}{\mathcal{C}}
\newcommand{\mA}{\mathcal{A}}

\newcommand{\mO}{\mathcal{O}}

\newcommand{\mM}{\mathcal{M}}
\newcommand{\mI}{\mathcal{I}}
\newcommand{\mJ}{\mathcal{J}}
\newcommand{\mF}{\mathcal{F}}

\newcommand{\mL}{\mathcal{L}}
\newcommand{\mS}{\mathcal{S}}
\newcommand{\mT}{\mathcal{T}}
\newcommand{\mEnd}{\mathcal{E}nd}

\newcommand{\bD}{\mathbb{D}}

\DeclareSymbolFont{AMSb}{U}{msb}{m}{n}
\DeclareMathSymbol{\boldk}{\mathord} {AMSb}{"7C}

\newcommand{\git}{/ \hspace{-3pt} / \hspace{-1pt}}
\begin{document}
\title{\addtocounter{footnote}{1} \'Etale Splittings of Certain Azumaya Algebras on Toric and Hypertoric Varieties in Positive Characteristic}
\author{Theodore J. Stadnik, Jr. \footnote{Author partially supported by NSF grant DMS-0636646}\\
	Northwestern University \\}
\date{\today}
\maketitle

\begin{abstract}
    \centering
    \begin{minipage}{0.6\textwidth}
    For a smooth toric variety $X$ over a field of positive characteristic, a $T$-equivariant \'{e}tale cover $Y \rightarrow T^*X^{(1)}$ trivializing the sheaf of crystalline differential operators on $X$ is constructed.  This trivialization is used to show that $\mD$ is a trivial Azumaya algebra along the fibers of the moment map $\mu: T^*X^{(1)} \rightarrow \mathfrak{t}^{*(1)}$.  This result is then extended to certain Azumaya algebras on hypertoric varieties, whose global sections are analogous to central reductions of the hypertoric enveloping algebra.  A criteria for a derived Beilinson-Bernstein localization theorem is then formulated.
    \end{minipage}
\end{abstract}

\tableofcontents

\section{Introduction}

\hspace{10pt} In the seminal article \cite{MVdB}, I. Mussen and M. Van den Bergh study central reductions of the hypertoric enveloping algebra in characteristic zero.  The following article will construct these central reductions in positive characteristic as the global sections of a sheaf of Azumaya algebras on the corresponding hypertoric variety.  It is shown that this Azumaya algebra is split along the fibers of the natural resolution of singularities $\mmM \rightarrow \mmM_0$.  This gives a hypertoric variant of a result by R. Bezrukavnikov, I. Mirkovic, and D. Rumynin (\cite[5.1.1]{BMR}).  The work of Braden-Licata-Proudfoot-Webster (\cite{PW}) and Bellamy-Kuwabara (\cite{BeK}) demonstrates there is a good theory of category $\mO$ and localization in characteristic 0.  Prompted by this fact, the formalism of \cite{BMR} is combined with the positive characteristic Grauert-Riemenschneider theorem to prove that a derived Beilinson-Bernstein theorem holds whenever the localization functor has finite homological dimension.  In order to accomplish these theorems, it is first necessary to consider the sheaf of differential operators on a smooth toric variety.\\

\hspace{10pt} The sheaf of crystalline differential operators, $\mD$, on a smooth variety $X/\mathbb{F}_p$ has traditionally been studied in number theory.  It was classically known that $\mD$ is a trivial Azumaya algebra when restricted to the zero section $X^{(1)} \rightarrow T^*X^{(1)}$.  This fact is extensively used in the study of Grothendieck's sheaf of differential operators $\bD$ (See \cite{Ha}). Until the appearance of \cite{BMR}, this was the most popular application of the theory.  In \cite{BMR}, it was shown that $\mD$ is an Azumaya algebra over the space $T^*X^{(1)}$ when $X$ is a smooth variety over an algebraically closed field of characteristic $p$.  They also construct a flat (but wildly ramified) cover $T^{*,1}X \rightarrow T^*X^{(1)}$ which trivializes $\mD$.  Based on a choice of a splitting for the Cartier operator, A. Ogus and V. Vologodsky (\cite{OV}) succeeded in trivializing $\mD$ on an \'etale cover.  The authors of \cite{BMR} also discovered that $\mD_{G/B}$ restricted to the fibers of the moment map $\mu: T^*{G/B}^{(1)} \rightarrow \mathfrak{g}^{*(1)}$ is trivial for $G$ semi-simple and all $p$ large enough.  In the same paper it is also shown there is an equivalence of categories $D^b(U(\mathfrak{g})^{\lambda}) \cong D^b(\mD^{\lambda}_{G/B})$ for regular weights $\lambda$, a derived version of Beilinson-Bernstein localization.\\

\hspace{10pt} This paper will look to generalize these results in the case of toric and hypertoric varieties, using only linear algebra and the construction of the flat cover.  In the case that $T \subset X$ is a toric variety, a $T$-equivariant \'etale cover $Y \rightarrow T^*X^{(1)}$ which $T$-equivariantly trivializes $\mD$ is constructed.  It is based upon the observation that the Euler operator $x\partial_x$ on $\mathbb{A}^1$ satisfies the Artin-Schreier equation (with $\mO_{T^*X^{(1)}}$-coefficients) $p(W)=W^p-W-x^p\partial^p$ and is $T$-invariant.  Its relation to the work of Ogus-Vologodsky is discussed in remark \ref{OVex}.  A slightly different description of $Y$ is used to proves that $\mD$ is trivial when restricted to the fibers of the moment map $\mu: T^*X^{(1)} \rightarrow \mathfrak{t}^{*(1)}$.  This construction is then used to create analogous \'etale covers splitting certain Azumaya algebras on hypertoric varieties.  The global sections of these Azumaya algebras are the characteristic $p$ analogues of the algebras studied in \cite{PW}, which the authors show have a good theory of category $\mO$.  Given that characteristic $p$ representation theory of a semi-simple Lie group loosely mimics category $\mO$ in characteristic $0$, it is reasonable to expect there to be a derived Beilinson-Bernstein localization theorem. \\

\hspace{10pt}The layout of this article is as follows. The third section contains basic linear algebra computations at points of $T^*{\mathbb{A}^n}^{(1)}$.  The linear algebra is organized so that most of the results from the subsequent subsections require little effort to prove.  These linear algebra computations are then used to prove a certain $T$-equivariant irreducible \'etale cover $Y \rightarrow T^*{\mathbb{A}^n}^{(1)}$ trivializes $\mD$.  This construction is then generalized to an arbitrary smooth toric variety and used to show $\mD$ is trivial on the fibers of the moment map.  In section four, Azumaya algebras on hypertoric varieties are constructed.  The global sections of these Azumaya algebras are the positive characteristic versions of the central reductions of the hypertoric enveloping algebra studied in \cite{PW}.  As in the previous section, an equivariant \'etale splitting is constructed and used to show triviality along the fibers of the symplectic resolution.  Finally, some results about a derived Beilinson-Berstein localization theorem are discussed.  The appendix contains characteristic-free proofs of well-known criteria for the behavior of subtorus actions on semistable points as well as some notes on the conventions for moment maps.

\section{Conventions}

Fix $\boldk$ an algebraically closed field of positive characteristic $p$.\\

Multi-index notation will be used. If $I \in \mathbb{Z}^{\oplus n}$ then $\partial^I = \partial_{x_1}^{i_1} ... \partial_{x_n}^{i_n}$ where $i_j$ is the $j$-th component of $I$.\\

Recall that for any scheme $X$ over $\mathbb{F}_p$, there is a natural map of $\mathbb{Z}$-schemes $F_{X/\mathbb{F}_p}: X \rightarrow X$ given by $\mO_X \rightarrow \mO_X: f \mapsto f^p$.  As $\boldk$ is an algebraically closed field of positive characteristic, $F_{\boldk/\mathbb{F}_p}$ is an invertible map.\\

If $X \rightarrow \boldk$ is a $\boldk$-scheme, then denote $X^{(1)}$ by the $\boldk$-scheme $X \rightarrow \boldk \shortstack[c]{$F_{\boldk/\mathbb{F}_p}^{-1}$\\ $\rightarrow$} \boldk$.  The natural map $F_{X/\mathbb{F}_p}$ is now a map of $\boldk$-schemes $X \rightarrow X^{(1)}$.  As $\boldk$ is fixed throughout this paper, denote it simply by $F_X$ or $F$.\\

For any closed point $x \in X$ denote the residue field at $x$ by $k(x)$.\\

\section{$\mD$-modules on toric varieties}
\subsection{Some linear algebra of Euler operators on $\mathbb{A}^n$}

In \cite{BMR}, is it is shown that on any smooth variety over $\boldk$, the sheaf of crystalline differential operators is an Azumaya algebra over its center.  Equivalently, the fiber of this sheaf at each closed point is a matrix algebra.  In this section, the action of the Euler operators on the point modules will be investigated in the case $X=\mathbb{A}^n$ and yield a slightly different description of fiber of $\mD$ at each point.\\

The main reference for the notation of this section is \cite{BMR}, a condensed review follows.

\begin{itemize}
 \item $\mD=\boldk  \langle \{x_i, \partial_i\} \rangle$ is the Weyl algebra of differential operators on $\mathbb{A}^n$.  Its center $Z(\mD)$ is $\boldk[\{x_i^p, \partial_i^p\}]$.
\item ${T^*\mathbb{A}^n}^{(1)} \cong Spec(Z(\mD))$ (\cite[1.3.2]{BMR}) and $T^{*,1}\mathbb{A}^n \cong Spec(\boldk[\{x_i, \partial_i^p\}])$ (\cite[2.1]{BMR}). $\mD$ is a sheaf of algebras over both of these spaces.
\item If $\zeta = (\vec{b},\omega) \in {T^*\mathbb{A}^n}^{(1)}$ a closed point and $\vec{a}$ the unique $p$-th root of $\vec{b}$ (component by component), then it lies under the closed point $\xi =(\vec{a},\omega) \in T^{*,1}\mathbb{A}^n$
\item For points as above, $\mD_{\zeta} = \mD \otimes_{Z} k(\zeta)$ and $\delta^{\xi} = \mD \otimes_{\boldk[\{x_i, \partial_i^p\}]} k(\xi)$, where in the latter formation of tensor product $\mD$ is given the right module structure over $T^{*,1}\mathbb{A}^n$.\\
\end{itemize}

As all spaces in this section are affine and all sheaves quasi-coherent, to ease notation there will be no distinction between a sheaf and its global sections.\\

\begin{lem}\cite[2.2.1]{BMR}\label{BMRlem} 
The natural left action of $\mD_{\zeta}$ on $\delta^{\xi}$ yields an isomorphism of algebras,
$$\mD_{\zeta} \cong End_k(\delta^{\xi})$$

Moreover, $\delta^{\xi}$ has a $\boldk$-basis $\{\partial^I\}_{I \in \{0,...,p-1\}^n}$ and the selected elements of $\mD_{\zeta}$ act on this basis in the following manner:

\begin{eqnarray*} (x_k -a_k) \cdot \partial^I &=& I_k  \partial^{I-e_k} \\ 
\partial_k \cdot \partial^I &=& 
\begin{cases} \partial^{I+e_k} & \text{if $I_k < p -1$} \\ \omega(\partial_k)^p  \partial^{I-(p-1)e_k} & \text{if $I_k = p-1$}
\end{cases}
\end{eqnarray*}

where $a_k$ is the $k$-th component of $\vec{a}$, $I_k$ the $k$-th component of $I$, and $e_k$ the standard $k$-th basis vector in $\mathbb{F}_p^n=\{0,...,p-1\}^n$.
\end{lem}

This lemma will be used to investigate the left action of the Euler operators $E_k=x_k \partial_k$ on $\delta^{\xi}$.  More concretely, to determine their characteristic and minimal polynomials.\\

\begin{df} The \textit{$(k,\tau)$ principal Euler block} (at $\xi=(\vec{a}, \omega)$), $T_k(\tau)$, is the $p \times p$ matrix defined by:\\
$$T_k(\tau) = \begin{bmatrix}
1-\tau 	& 0 			& 0 			& \cdots & 0 			& a_k \omega(\partial_k)^p 	\\
a_k 			& 2-\tau 	& 0 			& \cdots & 0 			& 0 												\\
0 			& a_k 			& 3-\tau 	&        & 0 			& 0 												\\
\vdots 	&  				& \ddots 	& \ddots & 					 & \vdots 										\\
\vdots  &         &         & \ddots & \ddots   & \vdots                \\
0				& 0				& 0				&  0		 & a_k			& -\tau											\\
\end{bmatrix}$$\\
\end{df}

\begin{pro}\label{matrixpro} Giving the basis $\{\partial^I\}_{I \in \{0,...,p-1\}}$ of $\delta^{\xi}$ the usual lexicographic ordering, left multiplication by $E_k-\tau$ on $\delta^{\xi}$ is represented in this ordered basis by the $p^{n-1} \times p^{n-1}$ block matrix:\\
$$ \widetilde{T_k(\tau)} = A_k \begin{bmatrix}
T_k(\tau) & 0 & \cdots & 0\\
0 &T_k(\tau) &  & 0\\
\vdots & & \ddots & \vdots \\
0 & 0 & \cdots & T_k(\tau) \\
\end{bmatrix} A_k^{-1}
$$\\
where $A_k^{-1}$ is some permutation matrix occurring from a reordering of the basis.\\
\end{pro}
\begin{proof}
\begin{eqnarray*}
(E_k - \tau)\cdot \partial^{I} & = & (x_k\partial_k - \tau)\partial^I\\
															 & = & [(x_k-a_k)\partial_k + a_k \partial_k - \tau]\partial^I\\
															 & = & \begin{cases} (x_k-a_k)\partial^{I+e_k} + a_k \partial^{I+e_k}-\tau \partial^I & \text{ if $I_k < p -1$}\\
															                     (x_k-a_k)\omega(\partial_k)^p \partial^{I-(p-1)e_k} + a_k \omega(\partial_k)^p\partial^{I-(p-1)e_k}-\tau \partial^I & \text{ if $I_k=p-1$}\end{cases}\\
															 & = & \begin{cases} (I_k+e_k)\partial^{I} + a_k \partial^{I+e_k}-\tau \partial^I & \hspace{117pt} \text{ if $I_k < p -1$}\\
															                      a_k \omega(\partial_k)^p\partial^{I-(p-1)e_k}-\tau \partial^I & \hspace{117pt} \text{ if $I_k=p-1$} \end{cases}\\
															 & = & 	\begin{cases} (I_k+e_k-\tau)\partial^{I} + a_k \partial^{I+e_k} & \hspace{128pt} \text{ if $I_k < p -1$}\\
															                      a_k \omega(\partial_k)^p\partial^{I-(p-1)e_k}-\tau \partial^I & \hspace{128pt} \text{ if $I_k=p-1$} \end{cases}\\													                    															                 
\end{eqnarray*}

The result is obvious in the case of $k=n$.  Let $A_k$ be the change of basis matrix described by the following:\\

Let $\sigma_k \hspace{-3pt}:\hspace{-3pt} \{0,...,p-1\}^{n} \rightarrow \{0,...,p-1\}^{n}$ be the permutation which interchanges the $n$-th component and the $j$-th component. Define a linear transformation by the rule $A_k(\partial^I)=\partial^{\sigma_k(I)}$.  $A_k$ takes the lexicographically ordered basis to the lexicographic basis where the $k$-th component has the least weight instead of the $n$-th component.

The result is now clear.\\
\end{proof}
\begin{cor}\label{charpoly}
If $\tau_k$ is a solution to the equation $T^p-T-a_k^p\omega(\partial_k)^p=0$, then the characteristic polynomial of left multiplication by $E_k-\tau_k$ on $\delta^{\xi}$ is

$$-(\lambda^p-\lambda)^{p^{n-1}}$$

and the minimal polynomial is 
$$\lambda^p-\lambda$$
\end{cor}
\begin{proof}
\begin{eqnarray*}
det(E_k-\tau_k - \lambda) &=& det(T_k(\tau)-\lambda)^{p^{n-1}} \hspace{80pt} \text{ (from \ref{matrixpro})}\\
										&=& (a_k\omega(\partial_k)^p a_k^{p-1}+\prod_{i=0}^{p-1}(i-\tau-\lambda))^{p^{n-1}} \hspace{5pt} \text{ (Expansion along last column)}\\
										&=& (a_k^p\omega(\partial_k)^p - (\lambda+\tau)^p + (\lambda+\tau))^{p^{n-1}}\\
										&=& (a_k^p\omega(\partial_k)^p - \tau^p + \tau - \lambda^p + \lambda)^{p^{n-1}} \\
										&=& -(\lambda^p-\lambda)^{p^{n-1}}
\end{eqnarray*}

$E_k-\tau_k$ satisfies the polynomial $\lambda^p-\lambda$.  Given the characteristic polynomial, this is the smallest possible equation that it can satisfy, thus it is the minimal polynomial.
\end{proof}

\begin{dfc}\label{Scon} Let $S = Z(\mD)[\{x_i\partial_i\}] \subset \mD$ (i.e., $S$ is the smallest subalgebra of $\mD$ generated by $Z(\mD)$ and the Euler operators) and set $Y_n =Spec(S)$. As ${T^*\mathbb{A}^n}^{(1)} = Spec(Z(\mD))$, there is a natural map $Y_n \rightarrow {T^*\mathbb{A}^n}^{(1)}$. Write a closed point $\eta \in Y_n$ over the point $\zeta = (\vec{b}, \omega) \in {T^*\mathbb{A}^n}^{(1)}$ as $\eta =(\vec{c}, \omega)$ where $c_k^p - c_k = b_k \omega_k = a_k^p \omega(\partial_k)^p$ (See \ref{netlem}).\\

Define $\mD_{\eta} = \mD \otimes_S k(\eta)$ where in the tensor product construction $\mD$ is again given its right $S$-module structure.\\
Observe that $\mD \otimes_Z S \otimes_S k(\eta) \cong \mD \otimes_Z k(\zeta) = \mD_{\zeta}$.  This implies $\mD_{\eta}$ has a natural left structure as an $\mD_{\zeta}$-module.  The next theorem will determine its module structure using the eigenvalues of the Euler operators.\\
\end{dfc}

\begin{thm}\label{isothm}
As left modules over $\mD_{\zeta}$, $\mD_{\eta} \cong \delta^{\xi}$, and has a unique $1$-dimensional subspace annihilated by the entire collection $\{E_k-c_k\}$.
\end{thm}
\begin{proof}
$\mD_{\eta} = \mD \otimes_{S} Z(\mD)[\{x_i\partial_i\}]/(\{x_i\partial_i - c_k\}) = \mD/\sum_{k=0}^n\mD(E_k-c_k)$.\\

Inducing the following exact sequence of left $\mD$-modules:

$$\mD^{\oplus n} \shortstack[c]{$\tilde{\phi}$ \\ $\rightarrow$} \mD \twoheadrightarrow \mD_{\eta} \rightarrow 0$$

where $\tilde{\phi}$ is the map $\vec{v}=(v_k) \mapsto \sum_k v_k(E_k-c_k)$.\\

Applying $- \otimes_{Z(\mD)} k(\zeta)$ to the above exact sequence, results in the following exact sequence of left $\mD_{\zeta}$-modules:

$$ \mD_{\zeta}^{\oplus n} \shortstack[c]{$\phi$ \\ $\rightarrow$} \mD_{\zeta} \twoheadrightarrow \mD_{\eta} \rightarrow 0$$

where the first arrow, $\phi$, is given by $\vec{v} \mapsto \sum_k v_k (E_k-c_k)$.\\

It is already known that $\mD_{\zeta}$ is a matrix algebra from \ref{BMRlem}, so $\mD_{\eta} \cong \delta^{\xi}$ if and only if $dim_{k(\zeta)}(\mD_{\eta})=p^n$. (Recall over any field, modules over a matrix ring are a direct sum of simple objects and up to isomorphism there is only one irreducible module (\cite[Ch. IX]{Hu}).)\\

Proceed by induction on $n$.

Base case:
If $n=1$, the map $\phi$ is given (in the standard ordered basis $\mmB$ for $End(\delta^{\xi}))$ by the $p \times p$ matrix:\\

$$ [ \phi ]_{\mmB} = \begin{bmatrix}T_1(c_1)^{T} & 0 & \cdots & 0 \\
													0								& T_1(c_1)^{T} & & 0\\
													\vdots & & \ddots & \vdots\\
													0 & 0 & 0 & T_1(c_1)^{T} 
					\end{bmatrix}$$

The rank of $\phi$ is $p^2-dim(Nul(\phi))=p^2-p$ as $T_1(c_1)$ has minimal and characteristic polynomial $X^p-X$ by \ref{charpoly}.\\

From this, it follows that $dim_{k(\zeta)}(\mD_{\eta}) = p^2-(p^2-p)=p$.

The statement about the dimension of the space annihilated by $E_1-c_1$ is clear from \ref{charpoly}.\\

Inductive step:  Suppose it is true for $\mathbb{A}^{n-1}$.
Let $\pi: T^*\mathbb{A}^n \rightarrow T^*\mathbb{A}^{n-1}$ be the natural projection forgetting the last coordinate.  There is a similar natural map $\pi_Y: Y_n \rightarrow Y_{n-1}$ by forgetting the last coordinate of $Z(\mD)$ and the last Euler operator.  Let $\eta' = \pi_Y(\eta)$ and $\zeta'=\pi(\zeta)$.\\

From \ref{matrixpro}, $E_n-c_n$ acting on $\delta^{\xi}$ has kernel of dimension $p^{n-1}$.  There is a natural inclusion $\mD_{\zeta'} \subset \mD_{\zeta}$ and $\mD_{\zeta'}$ acts naturally on $Ker(E_n-c_n)$.  By dimension and inductive hypothesis, this representation is isomorphic to $\mD_{\eta'}$. Also by inductive hypothesis, this space has only a one-dimensional subset annihilated by the entire collection $\{E_k-c_k\}_{k=1}^{n-1}$.  Therefore, this space is the unique $1$-dimensional subspace of $\delta^{\xi}$ annihilated by the entire collection $\{E_k-c_k\}_{k=1}^n$.\\

The only conclusion left is to show that $dim_{k(\zeta)}(\mD_{\eta}) = p^n$.\\

The operators $E_k-c_k$ acting on $\mD_{\zeta}$ on the right are commuting and semi-simple with all eigenvalues in $\mathbb{F}_p$.  Thus, it is easy to see that the image of $\phi$ consists of the space spanned by the multi-eigenvectors not of multi-eigenvalue $(0,0,...,0)$ in $\mD_{\zeta}$.  It follows that the dimension of $\mD_{\eta}$ is the dimension of the multi-eigenspace of multi-eigenvalue $(0,0,...,0)$ inside of $\mD_{\zeta}$.\\

Decompose $\mD_{\zeta} \cong \oplus_{i=1}^{p^n} (\delta^{\xi})^{\checkmark}$ as an isomorphism of right $\mD_{\zeta}$-modules.  It has already been shown that $\partial^{\xi}$ has a unique $1$-dimensional subspace of multi-eigenvalue $(0,0,...,0)$. The right action on $(\partial^{\xi})^{\checkmark}$ is given by transposing the left action, whence $(\partial^{\xi})^{\checkmark}$ has a unique $1$-dimensional subspace of multi-eigenvalue $(0,...,0)$.  Thus, the dimension of the multi-eigenspace of multi-eigenvalue $(0,...,0)$ of $\mD_{\zeta}$ is $p^n$, which implies that $dim_{k(\zeta)}(\mD_{\eta})$ is $p^n$.\\
\end{proof}

\begin{cor}\label{loccor}
The natural action map induces an isomorphism,

$$\mD_{\zeta} \cong  End_{k(\eta)}(\mD_{\eta})$$
\end{cor}
\begin{proof}
$\mD_{\eta} \cong \delta^{\xi}$ as left $\mD_{\zeta}$-modules by \ref{isothm} and $\mD_{\zeta} \cong End_{k(\zeta)}(\delta^{\xi})$ by \ref{BMRlem}.  Observing that the natural map $k(\zeta) \rightarrow  k(\eta)$ is an isomorphism, one arrives at the conclusion.\\
\end{proof}

\subsection{An equivariant \'{e}tale splitting of $\mD_{\mathbb{A}^n}$}

Notice that $\mD$ has no non-zero zero divisors and is of rank $p^{2n}$ over its center.  Therefore it is not a trivial Azumaya algebra for $n>0$; as any matrix ring of rank greater than $1$ must have nilpotent elements.  The notation will remain the same as found in the previous section.  A brief review of the theory of Azumaya algebras follows, the standard references for this section are \cite{Milne} (for schemes) and \cite{DI} (for rings).\\

Let $X$ be a Noetherian scheme and $\mB$ a sheaf of central $\mO_X$-algebras.  $\mB$ is an \textit{Azumaya algebra} if it is locally free of finite rank over $\mO_X$ such that the natural map given by the left and right actions of $\mB$ on itself, $\mB \otimes_{\mO_X} \mB^{opp} \rightarrow \mEnd_{\mO_X}(\mB)$, is an isomorphism. In \cite{Milne} it is shown that this condition is equivalent to $\mB \otimes_{\mO_X} k(x)$ being a central simple algebra over $k(x)$ for all $x \in X$. In particular, if $X$ is finite type over an algebraically closed field, then $\mB$ is an Azumaya algebra if and only if $\mB \otimes_{\mO_X} k(x)$ is (isomorphic to) a matrix algebra for all closed points $x$.\\

Let $\mF$ be a locally free sheaf of finite rank on $X$, then $\mEnd_{\mO_X}(\mF)$ is an Azumaya algebra, since it is clearly a central simple algebra when specialized to each point of $X$. These are the simplest types of Azumaya algebras.  An Azumaya algebra $\mB$ is called \textit{trivial} if there is some locally free sheaf $\mF$ such that $\mB \cong \mEnd_{\mO_X}(\mF)$ as $\mO_X$-algebras.  The locally free sheaf $\mF$ is called a \textit{splitting bundle} for $\mB$.  A map $\pi: X' \rightarrow X$ is said to \textit{trivialize} $\mB$ if $\pi^*\mB$ is a trivial Azumaya algebra over $X'$.  It is a theorem that every Azumaya algebra is locally trivial in the \'etale topology \cite[Ch IV]{Milne}.\\

The following proposition will be used to simplify arguments.\\
\begin{pro}\label{nakpro}
Let $(R,\mathfrak{m})$ be a Noetherian local ring and $\phi$ a surjective map of $R$-modules $\phi: F_1 \twoheadrightarrow F_2$. If $F_1$ and $F_2$ are free of rank $r$, then $\phi$ is an isomorphism.  In particular, a map of finite rank free $R$-modules is an isomorphism if and only if it is an isomorphism mod $\mathfrak{m}$.
\end{pro}
\begin{proof}
By freeness, $Tor^1_R(F_2,k(\mathfrak{m})) = 0$, so the following sequence is exact.

$$0 \rightarrow Kern(\phi) \otimes_R k(\mathfrak{m}) \rightarrow F_1 \otimes_R k(\mathfrak{m}) \rightarrow F_2 \otimes_R k(\mathfrak{m}) \rightarrow 0$$

However, $\phi$ is surjective.  Any surjective map between vector spaces of the same dimension is always injective.  Thus, $Kern(\phi) \otimes_R k(\mathfrak{m}) = 0$ and by Nakayama's lemma $Kern(\phi)=0$.\\

The second statement follows from Nakayama's lemma also.  Suppose $\psi: G_1 \rightarrow G_2$ is a map surjective mod $\mathfrak{m}$. As $\psi$ is a surjection mod $\mathfrak{m}$, it must be a surjection.  $\psi$ an isomorphism mod $\mathfrak{m}$ implies $G_1$ and $G_2$ have the same rank.  By the previous result, $\psi$ is an isomorphism.\\
\end{proof}

\begin{lem}\label{netlem} The map $Y_n \rightarrow {T^*\mathbb{A}^n}^{(1)}$ (see \ref{Scon}) is an \'etale cover.  It isomorphic over ${T^*\mathbb{A}^n}^{(1)}$ to the Artin-Schreier cover $Spec(Z(\mD) \rightarrow Z(\mD)[T_1,...,T_n]/(\{T_i^p-T_i - x_i^p \partial_i^p\}))$.
\end{lem}
\begin{proof} Recall that the Euler operators satisfy the equation $(x_k \partial_k)^p - x_k \partial_k = x_k^p \partial_k^p$ in $\mD$.  Let $R = Z(\mD)[T_1,..,T_n]$, then there is a natural map $R \rightarrow S$ given by $T_i \mapsto x_i\partial_i$.\\

This gives a surjection of $Z(\mD)$-algebras, $Z(\mD)[T_1,...,T_n]/(\{T_i^p-T_i - x_i^p \partial_i^p\}) \twoheadrightarrow S$.\\

As both sides are free $Z(\mD)$-modules of the same rank, this is an injection by \ref{nakpro}.

The map $Z(\mD) \rightarrow Z(\mD)[T_1,...,T_n]/(\{T_i^p-T_i - x_i^p \partial_i^p\})$ is clearly \'etale.\\
\end{proof}

\begin{pro} Let $S$ be as in \ref{Scon} then $\mD$, considered as a right $S$-module, is locally free.
\end{pro}
\begin{proof} By \cite[II,8.9]{Har}, a module is locally free over an irreducible variety if and only if its fiber over the generic point and the closed points have the same vector space dimension.  Over the generic point, $\mD$ has dimension $p^{2n}$ over $Quot(Z(\mD))$.  $Quot(S)$ is $p^n$-dimensional over $Quot(Z(\mD))$, it follows that $\mD$ has dimension $p^n$ over $Quot(S)$.  Theorem \ref{isothm} confirms that the fiber over the closed point has dimension $p^n$.
\end{proof}

\begin{lem}\label{amainlem}
The action of $\mD \otimes_{Z(\mD)} S$ on $\mD$, considered as a left $\mD$-module and a right $S$-module, induces an isomorphism:\\
$$\mD \otimes_{Z(\mD)} S \cong End_{S}(\mD)$$
\end{lem}
\begin{proof}
Both sides of the equation are locally free sheaves of the same rank over $S$, using \ref{nakpro} it is only necessary to show it is an isomorphism after reduction to each point $\eta$ of $Y_n$.\\

By the change of rings formula, this is the same as showing $\mD_{\zeta} \cong End_{k(\eta)}(\mD_{\eta})$, which is the content of \ref{loccor}.\\ 
\end{proof}

\begin{thm} The map $Y_n \rightarrow {T^*\mathbb{A}^n}^{(1)}$ is a $T$-equivariant \'{e}tale cover $T$-equivariantly trivializing the Azumaya algebra $\mD$.
\end{thm}

\subsection{The case of an arbitrary toric variety}
Let $T$ be a connected torus over $\boldk$.  A toric variety is an irreducible separated $T$-variety, $X$, such that for some closed point $x \in X$, the restricted action map $T \times \{x\} \rightarrow X$ is an open immersion with (Zariski) dense image.\\

A basic result is that every normal toric variety can be covered by $T$-invariant affine open subsets.  Moreover, it can be shown these affine patches are isomorphic as $T$-varieties to $\mathbb{A}^{m} \times \mathbb{G}_m^{n-m}$ where $n$ is the rank of $T$.  This statement reduces all of our computations for general toric varieties to the case of $T$-invariant open subsets of $\mathbb{A}^n$.\\

$X$ will now always denote a smooth toric variety over $\boldk$ of dimension $n$.\\

Let $\mD =\mD_X$ be the ring of crystalline differential operators on $X$.  For every affine open $T$-stable subset $U \subset X$, $\mD(U)$ is a rational $T$-module.  This follows from the observation that $\mT_{X}(U)$ and $\mO_X(U)$ are rational $T$-modules, and these generate $\mD(U)$ as an algebra.  This $T$-module structure is compatible with sheaf restriction maps because the $T$-actions on $\mT_X(U)$ and $\mO_X(U)$ are also.\\

For each $T$-stable open affine set $U$, $T$ being linearly reductive implies there is a decomposition $\mD(U) = \oplus_{\chi \in X^*(T)} \mD(U)^{\chi}$, where $\mD_X(U)^{\chi}$ is the space on which $T$ acts through $\chi$.  Since the action of $T$ is compatible with restriction maps, each of the modules $\mD(U)^{\chi}$ glue together to make a sheaf $\mD^{\chi}$ on $X$.\\

The space of invariants $\mT_X(U)^T$ is $n$-dimensional over $\boldk$, which follows directly from the description of $U$ as being $T$-isomorphic to an open subset of $\mathbb{A}^n$.\\

The following sheaf of algebras over $T^*X^{(1)}$ will be the main focus of interest.
\begin{df}
$\mS = \oplus_{\chi \in pX^*(T)} \mD^{\chi} = \oplus_{\chi \in X^*(T)} \mD^{\chi^{p}}$.\\
\end{df}

\begin{pro}\label{torpro} Let $U$ be a $T$-stable affine open subset, then $\mS(U) \cong \mO_{T^*X^{(1)}}(U)[\{E_i\}]$ where $\{E_i\}$ is any basis for the space of $T$-invariant elements in $\mT_X(U)$.
\end{pro}
\begin{proof}
$U$ is $T$-isomorphic to an open subset of $\mathbb{A}^n$ so it is enough to prove this result for $X = \mathbb{A}^n$ with the usual $T$-action.\\

By \cite[1.3.2]{BMR}, $Z(\mD)(U)= \mO_{T^*X^{(1)}}(U) \subset \mS(U)$.  Therefore $\mS(U)$ is generated by $Z(\mD)(U)$ and differential expressions invariant under $T$, this is precisely the ring generated the Euler operators $\{E_i\}$ and $Z(\mD)(U)$.\\
\end{proof}

\begin{cor} The sheaf of algebras $\mS$ is commutative, contains $Z(\mD)$ as a subalgebra, and is locally free of rank $p^n$ over $Z(\mD)$.\\
\end{cor}

\begin{cor} $Y=\underline{Spec}(\mS) \rightarrow Spec(Z(\mD))=T^*X^{(1)}$ is an \'{e}tale cover.\\
\end{cor}
\begin{proof}
Since this is a local property, by \ref{torpro}, it is enough to investigate the extension $Y_n \rightarrow {T^*\mathbb{A}^n}^{(1)}$, which is an \'{e}tale cover (\ref{Scon}).\\
\end{proof}

\begin{thm}\label{emainthm}(Main Theorem 1) $Y$ is an $T$-equivariant \'{e}tale cover of $T^*X^{(1)}$ which $T$-equivariantly trivializes $\mD$ with splitting bundle $\mD$.
\end{thm} 
\begin{proof}
This statement is local, and thus, follows directly from \ref{amainlem}.\\
\end{proof}

\begin{rem}\label{OVex} Independently and earlier, Ogus and Vologodsky (\cite{OV}) had constructed this \'etale cover.  Their technique is rather general and does not use the assumption of a group action.  The presentation above emphasizes group actions, and this emphasis will be essential later.  The relationship between the two approaches is given by the canonical lift of Frobenius.  Every toric variety has a canonical lift of Frobenius to the length two Witt vectors given by taking the coordinate functions to the $p$-th powers.  In \cite{OV}, it is shown how to associate to a lift of Frobenius a splitting of the Cartier operator.  This splitting is the data the authors used to construct an \'etale cover which, in the case of a toric variety, is isomorphic over $T^*X^{(1)}$ to the cover constructed in this section.
\end{rem}
\subsection{Another description of $Y$ and triviality along fibers of the moment map}

\begin{pro} Let $T[p]$ be the kernel of the group map $T \rightarrow T$: $t \mapsto t^p$.  For each $T$-stable open subset $U$ of $X$, $\mD(U)$ is a naturally a $T[p]$-module.  The proof of \ref{torpro} shows that $\mS(U) = \mD(U)^{T[p]}$, the $T[p]$ invariants of $\mD(U)$.\\
\end{pro}

If $\Theta$ is a vector field on a variety over a field of characteristic $p$, then $\Theta^p(f \cdot g) = f \Theta^p(g) + g \Theta^p(f)$ for all regular functions $f$ and $g$.  Therefore, there is a natural way to think of $\Theta^p$ as giving a derivation, this derivation is denoted $\Theta^{[p]}$.\\

\begin{df} Let $\mathfrak{t}$ be the Lie algebra of $T$.  The map $\mathfrak{t} \rightarrow Sym(\mathfrak{t}) : \Theta \mapsto \Theta^p - \Theta^{[p]}$ is a $p$-linear map,  or equivalently, a $\boldk$-linear map $\mathfrak{t}^{(1)} \rightarrow Sym(\mathfrak{t})$.  This map determines a map of $\boldk$-algebras, $Sym(\mathfrak{t}^{(1)}) \rightarrow Sym(\mathfrak{t})$.  Following the notation of \cite{BMR}, denote the induced map of varieties by $AS: \mathfrak{t}^* \rightarrow \mathfrak{t}^{*(1)}$.\\
\end{df}

\begin{rem}The map $\Theta \mapsto \Theta^{[p]}$ can also be thought of as part of the data associated to a restricted Lie algebra (see \cite{Jantzen}).\\
\end{rem}

\begin{lem}\label{cart} The space $Y$ fits into a Cartesian square,
 
$$\xymatrix{ Y \ar@{->}[r]^{\pi} \ar@{->}[d]^{\mu_Y} & T^*X^{(1)} \ar@{->}[d]^{\mu} \\
 		  \mathfrak{t}^* \ar@{->}[r]^{AS} & \mathfrak{t}^{*(1)} }$$	

where $\pi$ is the covering map and $\mu$ is the moment map.  In particular, $Y \cong T^*X^{(1)} \times_{\mathfrak{t}^{*(1)}} \mathfrak{t}^*$.
\end{lem}
\begin{proof}
The moment map can be thought of as being entirely induced by the map $\mathfrak{t} \rightarrow \mD_{X}$, via the natural action of $\mathfrak{t}$ as vector fields on $\mD_X$ (\cite[Section 1.1.4]{CG} or section 5 of this paper).  This description shows that the construction of the moment map is compatible with the gluing of the space $T^*X$ via the affine patches $T^*U$ for $U \subset X$ $T$-stable open affine.  Below it is shown that the image of this map lands in $\mS(U)$ for all $T$-stable open affine $U$.  Thus, the image lands in $\mS(X)$.  This permits definition of the map $\mu_Y^*$ as the induced map $Sym(\mathfrak{t}) \rightarrow \mS(X)$.  Since $\mathfrak{t}$ is affine, there is a corresponding map of varieties $\mu_Y: Y \rightarrow \mathfrak{t}^*$.\\

The above discussion also shows that it is enough to consider only the case when $X$ is an open subset of $\mathbb{A}^n$.  Upon further inspection, the statement will hold for arbitrary open subsets of $\mathbb{A}^n$ if and only if it holds for the case $X = \mathbb{A}^n$.\\

In this setting,
$$\begin{array}{ll}
Sym(\mathfrak{t}) \cong \boldk[T_1,...,T_n] & Sym(\mathfrak{t}^{(1)}) \cong \boldk[T_1^p,...,T_n^p] \\
\mu^*(T_i^p) = x_i^p \partial_i^p & \mu_Y^*(T_i)=x_i\partial_i \\
AS^*(T^p_i)=T_i^p-T_i & \mO_{T^*X^{(1)} \times_{\mathfrak{t}^{*(1)}} \mathfrak{t}^*} \cong Z(\mD)[T_1,...,T_n]/(\{T_i^p-T_i - x_i^p \partial_i^p\})\\
\end{array}$$
Now compare with \ref{netlem} to obtain the result.\\
\end{proof}

\begin{thm}\label{smainthm}
The Azumaya algebra $\mD$ is trivial when restricted to the fibers of the moment map.
\end{thm}
\begin{proof}
Let $h \in \mathfrak{t}^{*(1)}$ and $h' \in \mathfrak{t}$ be such that $AS(h') = h$.\\

Consider that $\mD \otimes_{Z(\mD)} \mS \otimes_{\mS} k(h') \cong \mD \otimes_{Z(\mD)} k(h)$.  The left hand side of this equation is a trivial Azumaya algebra, so thus must be the right hand side.\\
\end{proof}

\section{Hypertoric varieties in positive characteristic}
The study of hypertoric varieties was initiated by R. Bielawski and A. Dancer (\cite{BD}), where they are called toric hyperkahler manifolds.  Their main appeal is that they are a hyperkahler version of toric varieties.  It turns out that the construction of the manifold (ignoring extra structures) is entirely algebraic in nature, even defined over $\mathbb{Z}$.  Therefore, it is natural to consider their reductions modulo $p$.  As they are varieties, it is also natural to investigate their structure after passing to the algebraic closure of $\mathbb{F}_p$.  By the work of \cite{PW} and \cite{BeK}, there is a good theory of modules and a Beilinson-Bernstein style localization in characteristic $0$.  This entices one to think about hypertoric varieties as being like the flag variety for a semi-simple Lie group.  Certain Azumaya algebras on these hypertoric varieties provide an avatar of the algebra $\mD_{G/B}^{\lambda}$.  In this section, hypertoric varieties in characteristic $p$ over an algebraically closed field will be defined, the Azumaya algebras constructed and investigated, and a localization theorem formulated. 

\subsection{The definition of a hypertoric variety}

Fix an integer $n$ and let $T = \mathbb{G}_m^n$.  There is a natural action of $T$ on $T^*\mathbb{A}^n$.  Let $\mmA$ be the data $(K,\alpha,\lambda)$ where $K \subset T$ is some connected subtorus, $\alpha \in X^*(K)$ is a character, and $\lambda \in \mathfrak{k}^*$.  Write $H$ for $T/K$ and $\mathfrak{h}$ for its Lie algebra.\\

The moment map with respect $K$, $\mu: T^*\mathbb{A}^n \rightarrow \mathfrak{t}^* \rightarrow \mathfrak{k}^*$, factors through the moment map with respect to $T$. See section 5 for more information on the conventions about moment maps.\\

Let $X$ be any $K$-variety and let $\mL_{X}$ be the trivial line bundle on $X$.  Any character $\chi$ of $K$ naturally determines an equivariant structure on $\mL_{X}$ and denote this equivariant line bundle $\mL_X(\chi)$.  Let $X^{\chi-ss}$ denote the set of semi-stable points of $X$ with respect $\mL_X(\chi)$.  Define $X\git_{\alpha} K = X^{\alpha-ss}\git K$ where the right-hand side is the usual GIT quotient.  It can be shown that $X\git_{\alpha}K$ is naturally isomorphic to $Proj(\oplus_{m \geq 0} \mO_X^{\alpha^{m}})$ where $\mO_X^{\alpha^m}$ is the submodule on which $K$ acts by character $\alpha^m$.  Refer to \cite{Mu} for further details on geometric invariant theory.\\

The hypertoric variety of $\mA$ is then defined as,

$$\mmM(\mmA) = \mu^{-1}(\lambda)\git_{\alpha}K$$

It is also commonly denoted $T^*\mathbb{A}^n\git \hspace{-2pt} \git_{(\lambda,\alpha)}K$ when working over $\mathbb{C}$.\\

Another variety to consider along with $\mmM(\mmA)$ is the categorical quotient $\mmM_0(\mmA) = \mu^{-1}(\lambda)/ K$.  Both of these schemes have a natural $H$ action.  The inclusion of the semi-stable points of the fiber into the fiber gives an $H$-equivariant map $v: \mmM(\mmA) \rightarrow \mmM_0(\mmA)$.  The deformation space $\widetilde{\mmM(K,\alpha)} = T^*\mathbb{A}^n \git_{\alpha}K$, which is a Lawrence toric variety, and its categorical quotient $\widetilde{\mmM_0(K)}= T^*\mathbb{A}^n \git K$ will also be used.  Notice that $\mmM(\mmA)$ and $\mmM_0(\mmA)$ define closed subvarieties of $\widetilde{\mmM(K,\alpha)}$ and $\widetilde{\mmM_0(K)}$ respectively.

\begin{df}\label{grdef} $\mmA$ is called \textbf{smooth} if $K$ acts freely on $\mu^{-1}(\lambda)^{\alpha-ss}$ and the quotient is smooth.  In which case, the map $v$ is a symplectic resolution.
\end{df}
See section five for combinatorial criteria on $\mmA$ to ensure it is smooth.

\begin{pro} If $K \neq T$ and $\mmA$ is smooth, then $\mmM(\mmA)$ and $\mmM_0(\mmA)$ are irreducible.
\end{pro}
\begin{proof}
Let $f=\prod_i x_i\xi_i$ then the set $U=D_f \cap \mu^{-1}(\lambda) \git_{\alpha} K$ is open and dense in $\mmM(\mmA)$.  As $\mmM(\mmA)$ is smooth, it is irreducible if and only if it is connected since every regular local ring is a domain.  Since $U$ is dense, $\mmM(\mmA)$ is connected if and only if $U$ is connected.  $U$ is (quasi-)compact and equipped with a surjective map $\mmM(\mmA) \rightarrow \mathfrak{h^*}$.  The fibers of this map are connected so each fiber must live entirely in one connected component of $U$, quasi-compactness implies the image of each connected component determines a closed subset of $\mathfrak{h^*}$.  They are all disjoint and $\mathfrak{h}^*$ is connected, so all but one is empty.
\end{proof}

\subsection{Descent and the functor of invariants}
Let $\mmA = (K, \alpha, \lambda^p-\lambda)$ be smooth data, then $\mmM(\mmA)^{(1)} = {\mu^{-1}(\lambda)}^{(1)}\git_{\alpha^p} K^{(1)}$.  It is well-known that since the action of $K^{(1)}$ is free on the $\alpha^p$-semi-stable points of ${\mu^{-1}(\lambda)}^{(1)}$, the category of sheaves on $\mmM(\mmA)^{(1)}$ is precisely the category of $K^{(1)}$-equivariant sheaves on the $\alpha^p$-semi-stable points of ${\mu^{-1}(\lambda)}^{(1)}$.  The equivalence is given by pulling back along the quotient map.  Its inverse is the functor of $K^{(1)}$-invariants.\\

This section aims to investigate $K$-equivariant sheaves on the $\alpha^p$-semi-stable points of ${\mu^{-1}(\lambda)}^{(1)}$ and the functor of $K$-invariants.  As $K$ does not act freely, the aforementioned equivalence needs modification.  To assist the stack-naive reader, the following presentation will only cover the case that the sheaves are supported on a closed $K^{(1)}$-orbit.  This case is easily seen to be equivalent to understanding the category of $K$-equivariant sheaves on $K^{(1)}$.  It will be shown that the category of $K$-equivariant sheaves on $K^{(1)}$ is equivalent to the category of $K[p]$-equivariant sheaves on a point.  Under this correspondence, the functor of $K$-invariants corresponds to the functor of $K[p]$-invariants.

\begin{lem}\label{invarlem} The category, $Coh_K(K^{(1)})$, of $K$-equivariant coherent sheaves on $K^{(1)}$ is equivalent to the category of finite dimensional algebraic representations of $K[p]$ via the functors:

$$F:V \mapsto (\mO_K \otimes_{\boldk} V)^{K[p]}$$
$$G:\mM \mapsto i_e^*\mM$$
\end{lem}
\begin{proof}
To clarify the functors above, $(\mO_K \otimes_{\boldk} V)^{K[p]}$ are the invariants with respect to the anti-diagonal action and is given the natural left $K$-equivariant structure.  $i_e: pt \rightarrow K^{(1)}$ is the inclusion of the identity and if $\phi$ is the $K$-equivariant structure on $\mM$ then $i_e^*\mM$ obtains $K[p]$-representation structure by restricting the action and projection maps to the fiber above $e$.

$G \circ F \cong id$ is a straight-forward computation.  To confirm that $F \circ G \cong id$, consider the natural map 
$$\mM \rightarrow a^*\mM \shortstack{$\phi$ \\ $\longrightarrow$} p_2^*\mM \cong \mO_{K} \otimes_{\boldk} \mM \rightarrow \mO_K \otimes_{\boldk} i_e^*\mM \supset (\mO_K \otimes_{\boldk} i_e^*\mM)^{K[p]}$$  The image of this map lands in the $K[p]$-invariants and is an $K$-equivariant map between $K$-acting transitively on $K^{(1)}$ implies that $G$ is conservative (i.e., reflects isomorphisms).  $G$ applied to this equivariant map is an isomorphism, therefore the map was an isomorphism.

\end{proof}
\begin{rem} For the reader comfortable with stacks, the previous lemma is simply descent under the isomorphism $[pt/K[p]] \cong [K^{(1)}/K]$.
\end{rem}

\begin{cor}\label{invarcor} $End^K_{\mO_{K^{(1)}}}(\mM) \cong End_{\boldk}^{K[p]}(i^*_e \mM)$ and $(i^*_e \mM)^{K[p]} \cong i^*_e\mM^K$.
\end{cor}

\subsection{The hypertoric enveloping algebra in positive characteristic}
The hypertoric enveloping algebra and its central reductions were first studied in \cite{MVdB} in characteristic $0$.  For a less general introduction, see section three of \cite{PW}.

In this section, fix smooth data $\mmA=(K,\alpha,\lambda)$ and write simply $\mmM$ for $\mmM(\mmA)$ and $\mmM_0$ for $\mmM_0(\mmA)$.\\
Similar to the situation considered in \ref{cart}, construct the following commutative diagram ($Y_n$ as in \ref{Scon}):

$$\xymatrix{ {T^*\mathbb{A}^n}^{(1)} \ar@{<-}[r] \ar@{->}[d]^{\mu^{\mathfrak{k}}} & Y_{\mathfrak{k}} =_{def} {T^*\mathbb{A}^n}^{(1)} \times_{\mathfrak{k}^{*(1)}} \mathfrak{k}^* \ar[d]^{\mu^{\mathfrak{k}}_{Y_{\mathfrak{k}}}} \ar@{<-}[r] & Y_n \ar@{->}[dl]^{\mu^{\mathfrak{k}}_{Y_n}}\\
							\mathfrak{k}^{*(1)}   \ar@{<-}[r]^{AS} & \mathfrak{k}^{*} & }$$

Since the context of $K$ is clear, the notation $\mathfrak{k}$ will be dropped and the vertical maps denoted simply by $\mu$, $\mu_Y$ and $\mu_{Y_n}$.\\
							
\begin{dfc} Recall the $K$-equivariant sheaf of algebras $\mD$ on $T^*\mathbb{A}^n$.  Define $\mA_{\lambda} = i^*_{\lambda}(\mD|_{\alpha-ss}^K)$ to be the (sheaf of) hypertoric enveloping algebras associated to $\mmA$, where $i^*_{\lambda}: \mu_Y^{-1}(\lambda) \git_{\alpha^p} K^{(1)} \rightarrow Y_{\mathfrak{k}} \git_{\alpha^p} K^{(1)}$.  It is a sheaf of algebras on $\mu_Y^{-1}(\lambda)^{(1)} \git_{\alpha^p} K^{(1)}$.  By the Cartesian property of the square in the above diagram, this space is naturally isomorphic to $\mmM^{(1)}$.  Call its global sections, $U_{\lambda} = \Gamma(\mmM^{(1)},\mA_{\lambda})$, the hypertoric enveloping algebra reduction associated to $\mmA$.  One can also consider the deformation algebra $\widetilde{\mA} = \mD|_{\alpha-ss}^K$ living on $\widetilde{\mmM}$.
\end{dfc}

\begin{rem} Symbolically, the global sections of $\widetilde{\mA}$ is the hypertoric enveloping algebra first considered in \cite{MVdB}, but in positive characteristic.  Let $Z_{HC}$, the (Harish-Chandra) subalgebra of $\mD(\mathbb{A}^n)^K$, be the subalgebra generated by the image of $Sym(\mathfrak{k})$. Any choice of $\lambda$ determines a central character $Z_{HC} \rightarrow \boldk$.  The global sections of the hypertoric enveloping algebra reduction $\mA_{\lambda}$ is given by $\mD(X)^K/(\mD(X)^K Ker(\lambda))$.  The latter is the central reduction considered in \cite{PW} and \cite{MVdB}.

\end{rem} 

\begin{pro}
The sheaf $\mA_{\lambda}$ is an $H$-equivariant sheaf of algebras.
\end{pro}

\begin{thm}\label{mth2}(Main Theorem 2)\label{mainthm} The sheaf $\mA_{\lambda}$ is an Azumaya algebra on $\mmM^{(1)}$.  It is $H$-equivariantly trivialized by the $dim(\mathfrak{h}^*)$-degree \'etale cover $\mmM_n=\mu_{Y_n}^{-1}(\lambda) \git_{\alpha} K^{(1)}$.
\end{thm}
\begin{proof}
One description (see \cite{Milne}) of a sheaf of Azumaya algebras on an algebraic variety is a sheaf of central algebras which is a matrix algebra at every (closed) point. To ease notation, write $\mA$ instead of $\mA_{\lambda}$.\\

Let $\pi$ always denote the appropriate GIT quotient map.  Let $\zeta \in \mu_{Y_n}^{-1}(\lambda)^{\alpha-ss},  \nu \in Y_{\mathfrak{k}}$ any point above $\zeta$ and $\eta \in Y_n$ any point above $\nu$.  Further suppose that $\zeta$ has a closed $K^{(1)}$-orbit in $\mu_{Y_n}^{-1}(\lambda)^{\alpha-ss}$.  For any point $p$, let $O(p)$ stand for the $K^{(1)}$ orbit.\\

It suffices to show that the natural left action yields an isomorphism of algebras,
\begin{eqnarray}
\mA_{\pi(\nu)} &\cong& End_{\boldk}(\mA_{\pi(\eta)})
\end{eqnarray}
for all choices of $\zeta$, $\nu$, and $\eta$.\\

$K$ being linearly reductive implies that $\mA_{\pi(\nu)} \cong \mD^K|_{\pi(\nu)} \cong (\mD|_{O(\nu)})^K$ and $\mA_{\pi(\eta)} \cong  (\mD|_{O(\eta)})^K$.\\

Hence, it suffices to show that $(\mD|_{O(\nu)})^K \cong End_{\boldk}((\mD|_{O(\eta)})^K)$.\\

From \ref{invarlem}, it follows that to give a $K$-equivariant sheaf on $O(\nu) \cong K^{(1)}$ is to give a $K[p]$-equivariant sheaf on a point.  It was shown this functor assigns $\mD_{\nu}$ to $\mD|_{O(\nu)}$ and $\mD_{\eta}$ to $\mD|_{O(\eta)}$ and in \ref{invarcor} it was checked that taking $K[p]$ invariants of the former is the same as taking the $K$-invariants of the latter.\\

This reduces $(1)$ to showing the natural map is an isomorphism,
\begin{eqnarray}
\mD_{\nu}^{K[p]} &\cong& End_{\boldk}(\mD_{\eta}^{K[p]}) \label{eqn: 2}
\end{eqnarray}
Appealing to the next lemma, the map generating the ideal defined by the fiber of $\lambda \in \mathfrak{k}^*$, $\Phi: \oplus_{i=1}^k \mD^{K[p]}_{\zeta} \rightarrow \mD^{K[p]}_{\zeta}$ has cokernel of dimension $p^{2dim(\mathfrak{h}^*)}$ as it is isomorphic to $\mD_{\nu}^{K[p]}$.\\

It is clear that the image of this map acts by $0$ on the module $\mD^{K[p]}_{\eta}$ and by the next lemma the module $\mD^{K[p]}_{\eta}$ has rank $p^{dim(\mathfrak{h}^*)}$. By the $\mathit{K[p]}$\textit{-equivariant} isomorphism from $\ref{loccor}$, the natural map $D_{\zeta}^{K[p]} \twoheadrightarrow End_{\boldk}(\mD_{\eta}^{K[p]})$ is surjective.  By dimension considerations, its kernel is the image of $\Phi$.  Now tensor with $k(\nu)$ and notice that $\Phi \otimes k(\nu) =0$.  By right exactness of tensor products equation $(2)$ holds.

The fact about the \'etale splitting is also precisely described by $(2)$.
\end{proof}
\begin{lem} In the notation of the above lemma,
$$dim_{\boldk}(\mD_{\zeta}^{K[p]}) = p^{n+dim{\mathfrak{h}^*}} \hspace{15pt} dim_{\boldk}(\mD_{\nu}^{K[p]}) = p^{2 dim{\mathfrak{h}^*}} \hspace{15pt} dim_{\boldk}(\mD_{\eta}^{K[p]}) = p^{dim\mathfrak{h}^*}$$
\end{lem}
\begin{proof}
Equip $\mD$ with its natural $T \times T$-equivariant structure as a module over $\mO_{{T^*\mathbb{A}^n}^{(1)}}$.
For $\tau \in X^*(T[p] \times T[p])$, $\mD_{\zeta}^{\tau}$ is precisely of dimension $1$.  For $\chi \in X^*(K[p])$, $\mD_{\zeta} = \oplus_{\tau \text{ over } \chi} \mD^{\tau}_{\zeta}$.  Thus, $dim_{\boldk}(\mD^{\chi}_{\zeta}) = p^{2n-dimK}=p^{n+dim{\mathfrak{h}^*}}$.  The first equality follows from the case when $\chi$ is trivial.\\

As $\mD_{\zeta}$ is free of rank $p^n$ over $S_{\zeta}$, $\mD_{\nu}$ is of rank $p^n$ over $S_{\nu}$.  Thus $dim_{\boldk}(\mD_{\nu}) = p^{n+dim\mathfrak{h}^*}$.  $\mD_{\zeta}$ is a projective $S_{\zeta}$-module since it is free.  It also has a $K[p]$-equivariant structure.  It is well-known that an equivariant module is equivariantly projective if and only if its underlying module is also projective. Thus, there exists $K[p]$-equivariant free $S_{\zeta}$ modules $\mF_1$ and $\mF_2$ such that $\mD_{\zeta} \oplus \mF_1 \cong \mF_2$ as equivariant $S_{\zeta}$-modules.  Let $r_1^{\chi}$ and $r_2^{\chi}$ denote the $S_{\zeta}$ rank of the $\chi$ weight space of $\mF_1$ and $\mF_2$ respectively.  Taking vector space dimensions of weight spaces, $p^{n+dim\mathfrak{h}^*}+r_1^{\chi}p^n = r_2^{\chi}p^n$.  Reducing this isomorphism to $\nu$ and $\eta$, one obtains the formulas $dim(\mD_{\nu}^{K[p]}) + r_1^{0}p^{dim\mathfrak{h}^*} = r_2^0p^{dim\mathfrak{h}^*}$ and $dim(\mD_{\eta}^{K[p]}) + r_1^{0} = r_2^0$.  Back substitution yields the latter two equalities.
\end{proof}


\begin{cor} For $\alpha$ such that $(K,\alpha,\lambda)$ is smooth for all $\lambda$, the sheaf $\widetilde{\mA}$ is an Azumaya algebra on $\widetilde{\mmM} \times_{\mathfrak{k}^{*(1)}} \mathfrak{k}$ with \'etale trivialization $\widetilde{\mmM} \times_{\mathfrak{t}^{*(1)}} \mathfrak{t}$.
\end{cor}

\begin{lem}\label{htorcar} The map $\mmM^{(1)} \leftarrow \mmM_Y$ fits into two Cartesian diagrams,
$$\xymatrix{ \mmM^{(1)} \ar@{->}[d] & \mmM_Y \ar@{->}[l] \ar[d]^{\mu_Y^{\mathfrak{h}}} \\
		\mathfrak{h}^{*(1)} & \mathfrak{h}^* \ar[l]^{AS} } \hspace{25pt}
\xymatrix{\mmM^{(1)} \ar@{->}[d] \ar@{<-}[r] & \mmM_Y \ar@{->}[d] \\ 
            \mmM^{(1)}_0 \ar@{<-}[r] & \mmM_{Y,0} }$$
\end{lem}
\begin{proof}
$Y_n$ is obtained from $T^*\mathbb{A}^n$ by solving Artin-Schreier equations with $T$-invariant coefficients.  Observe that $\mO_{\mu^{-1}_{Y_n}(\lambda)}^{\alpha^m} = \mO_{\mu^{-1}_{Y_{\mathfrak{k}}}(\lambda)}^{\alpha^m}[S_1,..,S_n]$ where $S_i$ denote the image of the generators from \ref{netlem}.

The geometric invariant theory constructions show that,
$$\begin{array}{ll}
 \mmM = Proj(\oplus_{m \geq 0} \mO_{\mu^{-1}_{Y_{\mathfrak{k}}}(\lambda)}^{\alpha^m}) & \mmM_Y = Proj(\oplus_{m \geq 0} \mO_{\mu^{-1}_{Y_{\mathfrak{k}}}(\lambda)}^{\alpha^m}[S_1,..,S_n]) \\
 \mmM_0 = \mO_{\mu^{-1}_{Y_{\mathfrak{k}}(\lambda)}}^{K} & \mmM_{Y,0} = \mO_{\mu^{-1}_{Y_{\mathfrak{k}}}(\lambda)}^{K}[S_1,..,S_n]\\
\end{array}$$
From this description, $\mmM_Y$ is simply obtained from $\mmM^{(1)}$ by either base change.

\end{proof}
\begin{thm} The Azumaya algebra $\mA_{\lambda}$ is trivial on the fibers of $\pi$ and $\mu^{\mathfrak{h}}$.
\end{thm}
\begin{proof} Follows from \ref{mth2} and \ref{htorcar}.	
\end{proof}

Finally, for the convenience of the reader a table relating the new objects to the classical objects studied in \cite{BMR} follows,

$$\begin{array}{lcl}
\widetilde{\mA} & \leftrightarrow & \widetilde{\mD}\\
\widetilde{\mmM} & \leftrightarrow & \widetilde{\mathfrak{g}^*}\\
\widetilde{\mmM_0^{(1)}} & \leftrightarrow & {\mathfrak{g}^*}^{(1)} \times_{{\mathfrak{h}^*}^{(1)} \git W} \mathfrak{h}^*\\
\Gamma(\widetilde{\mA}) & \leftrightarrow & \widetilde{U(\mathfrak{g})}\\
\end{array}
\hspace{25pt}
\begin{array}{lcl}
\mA_{\lambda} & \leftrightarrow & \mD^{\lambda}\\
U_{\lambda} & \leftrightarrow & U(\mathfrak{g})^{\lambda}\\
\mmM(K,\alpha,0) & \leftrightarrow & \widetilde{\mathcal{N}}\\
\mmM_0(K,\alpha,0) & \leftrightarrow & \mathcal{N}\\
\end{array}$$

\subsection{A localization theorem}

For hypertoric varieties in characteristic $0$ , there is a version of the Beilinson-Bernstein equivalence given by Bellamy and Kuwabara \cite{BeK}.  In characteristic $p$, no statement this strong is even true.  However, combined with the work of Bezrukavnikov-Mirkovic-Rumynin and Braden-Licata-Proudfoot-Webster, there is a strong suggestion that a derived Beilinson-Bernstein equivalence will hold. The main theorem of \cite{BMR} states that there is an equivalence of triangulated categories $D^{b}(Mod_{f.g}(U^{\lambda})) \rightarrow D^{b}(Mod_{l.f.g}(\mD^{\lambda}_{G/B}))$ for regular integral weights $\lambda \in \mathfrak{h}^*$.  This equivalence is given by derived localization functor $V \mapsto V \otimes_{U^{\lambda}}^{\mathbb{L}} \mD_{G/B}^{\lambda}$ with inverse functor given by the derived global sections functor.  A key ingredient using the regularity assumption on $\lambda$ is showing that the localization functor has finite homological dimension.  This section will show that this is the only obstruction to obtaining a localization theorem on hypertoric varieties. It will be necessary to use some results employing the technique of Frobenius splittings. Recall a Frobenius splitting is an $\mO_{X^{(1)}}$ splitting of the natural map $\mO_{X^{(1)}} \rightarrow F_*\mO_{X}$.  For more information about Frobenius splittings, see \cite{BK}.

\begin{thm}\label{grpc}(Grauert-Riemenschneider in positive characteristic)\cite{MK}[1.2] \cite{BK}[1.3.14]
Let $\pi: X \rightarrow Y$ be a proper birational map of varieties a perfect field of characteristic $p > 0$, such that:
\begin{enumerate}
\item $X$ is smooth and there is $\sigma \in \Gamma(X, \omega_X^{-1})$ such that $\sigma^{p-1}$ Frobenius splits $X$
\item $D = div(\sigma)$ contains the set-theoretic exceptional locus of $\pi$
\end{enumerate}
Then $\mathbb{R}^i\pi_*(\omega_X) = 0$ for all $i>0$.
\end{thm}

\begin{cor} $\mathbb{R}\Gamma(\mmM(\mmA)^{(1)},O_{\mmM(\mmA)}) \cong \Gamma(\mmM(\mmA)^{(1)},\mO_{\mmM(\mmA)})$
\end{cor}
\begin{proof}
There is a natural splitting of Frobenius on $T^*\mathbb{A}^n$ given by (in multi-index notation) $x^{p-1}\xi^{p-1}\frac{1}{(dx\xi)^{p-1}}$.  It is clearly a $(p-1)^{st}$ power.  This splitting is the standard splitting realizing $\mO_{T^*\mathbb{A}^n}$ as a $T[p] \times T[p]$-graded module over its Frobenius twist.  This splitting is $T$-equivariant and the (support of the) associated divisor is the closed subscheme $F = \cup_i  \{ x_i \xi_i = 0 \}$.\\

Let $\widetilde{D} = F \git_{\alpha} K$.  Since the Frobenius morphism respects the $K$-action (in an appropriate sense), $\widetilde{\mmM}$ is Frobenius-split by a $(p-1)^{st}$ power whose associated divisor is the closed subscheme $\widetilde{D}$.  The map $\pi: \widetilde{\mmM} \rightarrow \widetilde{\mmM_0}$ is birational with exceptional locus contained in $\widetilde{D}$.  Thus the axioms of $\ref{grpc}$ are satisfied, yielding $$\mathbb{R}\Gamma(\widetilde{\mmM(\mmA)}^{(1)}, \mO_{\widetilde{\mmM(\mmA)}}) = \mathbb{R}\Gamma(\widetilde{\mmM(\mmA)}, \mO_{\widetilde{\mmM(\mmA)}}) \cong \mathbb{R}\pi_*\omega_{\widetilde{\mmM}} \cong \Gamma(\widetilde{\mmM(\mmA)}, \mO_{\widetilde{\mmM(\mmA)}})$$

One can check that the ideal describing $\mmM(\mmA)$ as a closed subset of $\widetilde{\mmM(\mmA)}$ can be generated by a regular sequence of elements.  Hence by induction and the long exact sequence in cohomology, the higher cohomology groups of $\mO_{\mmM(\mmA)}$ must also vanish.
\end{proof}

\begin{cor}\label{vanish} When the data $(K,\alpha,0)$ is smooth, $\mathbb{R}\Gamma(\mmM^{(1)},\mA_{\lambda}) \cong U_{\lambda}$
\end{cor}
\begin{proof} $\mA_{\lambda}$ comes naturally with a canonical filtration.  The construction is given by using a certain filtration on $\mD$.  Consider that $\mD$ has the ordered basis $\{x^I\partial^J\}_{I,J \leq \mathbf{p-1}}$ over its center and define a length $p^n-1$ filtration by setting $F_i$ to be the subset of generated by all basis elements with $\sum I \leq i$.  This is an analogue of the standard filtration on $\mD$, but now the filtration is by coherent modules over the center.  Use the usual convention that $F_{-1}=0$.  

Notice that $Gr(\mD) \cong \mO_{T^*\mathbb{A}^n}$.  Equip $\mD|_{\alpha-ss}^K$ with the descended filtration.  Giving the ideal generated by the moment map the induced filtration and the quotient, $\mA_{\lambda}$ the cokernel filtration, there is a strictly filtered sequence of modules:  $0 \rightarrow \mJ \rightarrow \mD|_{\alpha-ss}^K \rightarrow \mA_{\lambda} \rightarrow 0$.  As the maps are strict, taking the associated graded gives an exact sequence, $0 \rightarrow Gr(\mJ) \rightarrow Gr(\mD|_{\alpha-ss}^K )\rightarrow Gr(\mA_{\lambda}) \rightarrow 0$. The functor of $K$-invariants is exact, therefore commuting with associated graded, obtaining $Gr(\mD|_{\alpha-ss}^K) \cong \mO_{\mmM^{(1)} \times_{\mathfrak{k}^{*(1)}} \mathfrak{k}^* }$.  Notice $Gr(\mJ)$ is the ideal describing the fiber above $0$.  Therefore, $Gr(\mA_{\lambda}) \cong \mO_{\mmM(K,\alpha,0)}$.\\

Using the previous corollary, $\mathbb{R}\Gamma(\mmM^{(1)}, \mO_{\mmM(K,\alpha,0)})$ is acyclic, thus each term $F_{i+1}/F_i$ has vanishing cohomology.  By induction and the long exact sequence in cohomology induced by the short-exact sequence $0 \rightarrow F_{i} \rightarrow F_{i+1} \rightarrow F_{i+1}/F_i \rightarrow O$, it follows that $\mathbb{R}\Gamma(\mmM^{(1)},F_{i+1})$ for all $i$.  In particular, when $i=p^n-1$, $F_i = \mA_{\lambda}$ and the corollary is proven.

\end{proof}

\begin{df} Define $Loc_{\lambda}: D^-(Mod_{f.g.}(U_{\lambda})) \rightarrow D^-(Mod_{l.f.g}(\mA_{\lambda}))$ by $Loc_{\lambda} = \mA_{\lambda} \otimes_{U_{\lambda}}^{\mathbb{L}} -$.
\end{df}

\begin{thm}\label{locthm} If the functor $Loc_{\lambda}$ has finite homological dimension, then there is an equivalence of triangulated categories
$$D^b(Mod_{f.g.}(U_{\lambda})) \cong D^b(Mod_{l.f.g}(\mA_{\lambda}))$$
given by $Loc_{\lambda}$ and $\mathbb{R}\Gamma$.
\end{thm}
\begin{proof}  First, observe that the case of $K=T$ is trivial so assume $K \neq T$.  With this assumption, the proof of the theorem is mostly an adaptation of the techniques from $\cite{BMR}$, therefore only a concise sketch is provided.  Using the filtration on $\mA_{\lambda}$ constructed above, one can create a nearly identical proof that $\mathbb{R}\Gamma$ sends chains of locally finitely generated $\mA_{\lambda}$-modules to chains of finitely generated modules for $U_{\lambda}$.  It is also easy to show that $\mathbb{R}\Gamma$ (derived in sheaves of $\mA_{\lambda}$-modules) respects forgetful functors.  Therefore, no distinction between deriving $\Gamma$ on $Coh(\mmM)$ and $Mod(\mA_{\lambda})$ will be made.\\

$Loc_{\lambda}$ and $\mathbb{R}\Gamma$ are clearly adjoint functors and it is obvious that $Loc(U_{\lambda}) \cong \mA_{\lambda}$.

Lemma $3.5.1.$ of \cite{BMR} has the following form in this case: $\mA_{\lambda}$ is an Azumaya algebra on the smooth connected variety $\mmM^{(1)}$, which has a projective morphism $\mmM^{(1)} \rightarrow \mmM^{(1)}_0$ and $\mmM_0^{(1)}$ is an affine variety.  $\mmM^{(1)}$ is Calabi-Yau because it is a symplectic variety, thus $D^b(Mod_{l.f.g}(\mA_{\lambda}))$ is a Calabi-Yau category.\\

The $K \neq T$ assumption used in $3.5.3$ of \cite{BMR}, permits use of $3.5.2$ of \cite{BMR}.  It is only left to show that the natural map $id \rightarrow \mathbb{R}\Gamma \circ Loc_{\lambda}$ is an isomorphism.  This is precisely the content of \ref{vanish}.
  
\end{proof}

\begin{rem} For generic choices of $\lambda$, the associated hypertoric variety is affine, so this theorem is obvious generically.  The most interesting case is an integral choice of $\lambda$ (i.e., $\lambda = d\Lambda$ for $\Lambda \in X^*(K)$), which satisfies $\lambda^p-\lambda=0$.  The hypertoric variety in this case is almost never affine.
\end{rem}

\begin{df} Let $W_I$ denote the wall of $I$ from \ref{walldf} but considered in $X^*(K)$.  For $I$ minimal with respect to the property  $dim(\mathfrak{k} \cap t_I)=1$ (i.e. $I$ is a circuit for the polytope associated to the hyperplane arrangement), let $\vec{n}_I$ denote a primitive normal vector to $W_I$ in $X_*(K) \subset X_*(T)$.  It can be chosen uniquely by requiring it to pair positively with $\alpha$.  We define $N_I = \sum_{i} |\langle \vec{n}_I, \chi_i \rangle_{T}|$, $N=max_I\{N_I\}$, and a polytope $P \subset X^*(K)$ by $P=\{\chi : | \langle \vec{n_I}, \chi \rangle_K | \leq N_I \}$
\end{df}

Combined with \ref{locthm}, the following two conjectures would give a localization theorem mirroring the one of Bezrukavnikov-Mirkovic-Rumynin.

\begin{con}\label{weakconj} If $p > N$,\
\begin{enumerate}
\item (Weak form) $\lambda = d \Lambda$ integral, and $\Lambda \notin P$ then $Loc_{\lambda}$ has finite homological dimension.
\item (Strong form) $K$ and $\alpha$ satisfying some positivity condition and $\vec{n}_I$ chosen to be positive, then it is enough to choose $\lambda$ not in the negative part of $P$ (with respect to the orientations given by the $\vec{n}_I$).
\end{enumerate}
\end{con}

\begin{rem} As shown in the examples below, often times it is not required to avoid the entire polytope $P$.  In the case of Lie groups, $\lambda$ must be chosen to be regular in the Cartan algebra, an algebra and an action defined by choice of $B =TN$.  In the hypertoric case, $\lambda$ is chosen independently of $\alpha$ ($\alpha$, in a loose sense, plays the role of a Borel subalgebra).  Conjecture $\ref{weakconj}$ anticipates that by manipulating $\alpha$ and $K$ with Fourier transform, it is possible to get a better estimate the region to avoid.  The partial Fourier transform can flip the signs occurring in the expression of $\alpha$ as a sum of the standard basis vectors $\chi_i$.  However, the partial Fourier transform is slightly non-symmetric, applied twice it is multiplication by $-1$.  In particular, Fourier transform is very non-symmetric with respect to the moment map.  It is believed that if the correct $\alpha$ and $K$ are chosen, the more precise combinatorial condition on $\lambda$ will have a nice form.
\end{rem}

\begin{exa} In the case, $p>2$, $K = (z,z) \subset \mathbb{G}_m^2$ and $\alpha = 1$, $Loc_{\lambda}$ has finite homological dimension for integral $\lambda \neq -1$.  This situation is exactly the one considered in \cite{BMR} for the case $G=SL_2$.
\end{exa}

There is also the following counter-example credited to A. Braverman in an earlier version of \cite{BMR},

\begin{exa} Let $p > 2$ and $K$ the diagonal subtorus of $T=\mathbb{G}_m^{p+1}$.  Choose $\alpha = 1$ and $\lambda =0$, then $\mmM(K,1,0) \cong T^*\mathbb{P}^p$. The module $\mO(-p)$ naturally has the structure of an $\mA_0$-module via the description $\mO(-p) \cong Fr^*\mO(-1)$.  However, $\mathbb{R}\Gamma(T^*\mathbb{P}^p, \mO(-p)) \cong 0$, so it is impossible to have the equivalence of categories in \ref{locthm}.
\end{exa}

\section{Appendix on moment maps and the combinatorics of hyperplane arrangements}
\subsection{Conventions on moment maps} 
Throughout this appendix $\langle -,-\rangle$ will denote the pairing between a space and its dual.\\

Let $X$ be a symplectic $G$-variety. That is, a symplectic variety with an action of $G$ preserving the symplectic form $\omega$.  Traditionally, the moment map $X \rightarrow \mathfrak{g}^*$ is defined via the following property with respect to $\omega$ and the $G$-action.

\begin{df} A moment map is a for $G$ acting on $(X,\omega)$ is a $K$-equivariant map $\mu:X \rightarrow \mathfrak{g}^*$ which satisfies:
\begin{enumerate}
\item $\mu$ is $G$ equivariant for $G$ acting on $\mathfrak{g}^*$ via the coadjoint action.
\item $d\langle\mu(-),\theta \rangle_x(\xi) = \omega_x(\xi,\theta_x)$ where $\theta_x$ is the vector field on $X$ determined by $\theta$ via the $G$-action and $\langle\mu(-),\theta \rangle$ is the function $X \rightarrow \mathbb{A}^1; m \mapsto \langle \mu(m),\theta \rangle$.
\end{enumerate}
\end{df}

In the case when $G = K$ is a torus, a moment map (as defined by these properties) is only unique up to adding an element of $\mathfrak{k}^*$.\\

Despite this ambiguity, in the case of the cotangent bundle, there is a canonical construction which will give us a moment map $T^*X \rightarrow \mathfrak{k}^*$ herein abusively called \underline{the} moment map.\\

Let $K$ act on an (affine) variety $X$ and let $\omega$ be the canonical symplectic form on $T^*X$.  $T^*X$ is naturally a symplectic $K$-variety.  The $K$ action on $X$ gives rise to a map $\mathfrak{k} \rightarrow Vect(X)$.  By universal properties, this extends to a map of algebras $Sym(\mathfrak{k}) \rightarrow \mD_{X}$.  Both sides come with canonical filtrations and this is a map of filtered rings, taking associated graded construction, there is a map $Sym(\mathfrak{k}) \rightarrow \mO_{T^*X}$.  Let $\mu: T^*X \rightarrow \mathfrak{k}^*$ be the spectrum of this map.

It turns out that this map $\mu$ has the property to be a moment map.  Therefore in the situation of a cotangent bundle described above, this will be referred to as \underline{the} moment map.\\

One can easily check that if $K' \subset K$ acting on $X$ then the moment map $T^*X \rightarrow \mathfrak{k}^{'*}$ is just given by the moment map for $K$ composed with the natural map $\mathfrak{k}^* \rightarrow \mathfrak{k}^{'*}$.\\

In the main document, there is a need to consider other moment maps.  First, the symplectic form on the variety $\mmM(\alpha,\lambda)$ is described.\\

\begin{pro} For any $m = \overline{(z,w)} \in \mmM(\alpha,\lambda)$
\begin{enumerate}
\item $T_{m}\mmM(\alpha,\lambda) \cong T_{(z,w)}\mu^{-1}(\lambda)/T_{(z,w)}O(z,w)$ induced by the natural map $T_{(z,w)}\mu^{-1}(\lambda) \rightarrow T_{m}\mmM(\alpha,\lambda)$.
\item $T_{(z,w)}\mu^{-1}(\alpha,\lambda) = (T_{(z,w)}O(z,w))^{\perp \omega} \subset T_{(z,w)}T^*\mathbb{A}^n$.
\end{enumerate}
where $O(z,w)$ denotes the $K$-orbit of the point $(z,w)$.
\end{pro}
\begin{proof}
To show the second statement one only needs the fact that $T_{(z,w)}O(z,w) = \{ \theta_{(z,m)} | \theta \in \mathfrak{k} \}$.  The first statement follows from the second.
\end{proof}

\begin{df} Let $\overline{\omega} : T_{m}\mmM(\alpha,\lambda) \times T_{m}\mmM(\alpha,\lambda) \rightarrow \boldk$ be given by descending $\omega$.  By the previous proposition, it follows that this $\overline{\omega}$ is well-defined and non-degenerate.  This is the symplectic form associated to $\mmM(\alpha,\lambda)$.  Notice that it is invariant for the action of $H=T/K$.
\end{df}

Fix $\lambda \in \mathfrak{k}^*$ and choose $\tilde{\lambda} \in \mathfrak{t}^*$, a lift of $\lambda$, then define a premoment map by the following equation:

$$(z,w) \mapsto v_{(z,w)} \text{ where } v_{(z,w)} \text{ is the unique solution to } \langle v_{(z,w)},a_i \rangle = z_i w_i + \langle \tilde{\lambda},e_i \rangle$$

and $a_i$ is the natural image of $e_i$ under $\mathfrak{t} \rightarrow \mathfrak{h}$.\\

This can be described as $Sym(\mathfrak{h}) \rightarrow \Gamma(T^*\mathbb{A}^n,\mO_{T^*\mathbb{A}^n})$ via the map sending $a_i \mapsto x_i\xi_i + \langle \tilde{\lambda},e_i \rangle$. This is the Hamiltonian for the action of $H$.\\

Passing to the GIT quotient, there is an $H$-equivariant map ${}_H\mu: \mu^{-1}(\lambda)\git_{\alpha}K \rightarrow \mathfrak{h}^*$ (for any $\alpha$)

It will be shown that this is a moment map for the action of $H$ and that all moment maps for the action of $H$ arise in this manner.  Assuming the former statement, the later statement is simple to prove.  All other moment maps come as shifts of this moment map by elements of $\mathfrak{h}^*$, which will be the same as the resulting moment map gotten by shifting $\tilde{\lambda}$ by the same element of $\mathfrak{h}^*$ (i.e. choosing a different lift of $\lambda$).  This subsection concludes by proving the former statement.\\

\begin{lem}
The map ${}_H\mu$ is a moment map for $H$ acting on $(\mmM(\alpha,\lambda),\overline{\omega})$.
\end{lem}
\begin{proof}
It has already been shown that it is $H$-equivariant.  Observing that for any $a_i \in \mathfrak{h}$, $\langle{}_H \mu(\pi(-)), a_i \rangle$ differs from the function $\langle { }_T \mu(-), e_i \rangle |_{\mu^{-1}(\lambda)^{\alpha-ss}}$ by only a constant,
\begin{eqnarray*}
d\langle{}_H\mu(\pi(-)),a_i \rangle_m & = & d(\langle{}_T\mu(-), e_i \rangle|_{K_\mu^{-1}(\lambda)^{\alpha-ss}})_{(z,w)}\\
		& = &     d\langle{}_T\mu(-), e_i\rangle_{(z,w)}|_{T_{(z,w)}\mu^{-1}(\lambda)^{\alpha-ss}}\\
		& = &  \omega(-,(e_i)_{(z,w)})\\
		& = &  \overline{\omega}(\pi(-),(a_i)_{m})\\
\end{eqnarray*}

\end{proof}
\subsection{Hyperplane arrangements}
In the section, criteria for triples $(K, \alpha, \lambda)$ to satisfy $\ref{grdef}$ will be created from the combinatorial data of hyperplane arrangements.\\

Many of the results and arguments of this section were originally due to \cite{BD} and \cite{Ko}.  In this appendix, purely algebraic proofs which work over a field of arbitrary characteristic are provided.\\

Given any triple $(K, \alpha, \lambda)$ the natural inclusion $K \subset T$ gives rise to an exact sequence of vector spaces:

$$0 \rightarrow \mathfrak{k} \rightarrow \mathfrak{t} \rightarrow \mathfrak{h} \rightarrow 0$$

and an exact sequence of abelian groups:

$$ 0 \rightarrow X_*(K) \rightarrow X_*(T) \rightarrow X_*(H) \rightarrow 0$$

where $\mathfrak{h}$ is the Lie algebra of $H=T/K$. \\

Make the following notational conventions: $h = dim(H)$, $n=dim(T)$ and $k=dim(K)$.\\

If $\{e_i\}_{i=0}^n$ (resp. $\{\chi_i\}_{i=0}^n$) denotes the standard basis of $\mathfrak{t}$ (resp. $X_*(T)$) then the map $\mathfrak{t} \rightarrow \mathfrak{h}$ (resp. $X_*(T) \rightarrow X_*(H)$) is determined by $e_i \mapsto a_i$ (resp. $\chi_i \mapsto A_i)$.  Notice that such a selection also determines the torus $K$.  Let $a_i^{\checkmark}$ (resp. $A_i^{\checkmark})$ denote the corresponding linear functional in $\mathfrak{h}^*$ (resp. character in $X^*(H)$).  Let $\tilde{\lambda} \in \mathfrak{t}^*$ be a lift of $\lambda$ and $\tilde{\lambda}_i = \langle \tilde{\lambda}, e_i \rangle$.  Set $H_i = \{ \vec{v} \in \mathfrak{h}^* | \langle \vec{v},a_i \rangle = \tilde{\lambda}_i \}$ and $\mA = \{H_i | 0 \leq i \leq n \}$.\\

Define $\mA$ to be simple if any collection of $k+1$ hyperplanes from $\mA$ has empty intersection.  Say that $\mA$ is smooth if is both simple and, for any collection $\mC \subset \mA$ of $k$ hyperplanes with non-trivial intersection, the set $\{A_i\}_{H_i \in \mC}$ is a $\mathbb{Z}$-basis for $X_*(H)$.

\begin{pro}\label{finitepro} Let $(z,w) \in \mu^{-1}(\lambda)$ then $K_{(z,w)}$ (the stabilizer of $(z,w)$ in $K$) is finite if and only if $\{ a_i | \mu((z,w)) \in H_i \}$ is linearly independent.
\end{pro}
\begin{proof}
Let $I=\{i | z_i=w_i=0 \}$.  It is not hard to check that the $T$-stabilizer of $(z,w)$ is the torus defined by the ideal generated by $\{ t^{e_i}-1 | i \notin I \}$.  Let this torus be denoted by $T_I$.  Now $Lie(K_{(z,w)})=Lie(T_I \cap K) = Lie(T_I) \cap Lie(K)$, in particular, $K_{(z,w)}$ is finite if and only if $Lie(T_I) \cap \mathfrak{k} = 0$.  By construction, $Lie(T_I)=\mathfrak{t}_I$ is generated by the set $\{e_i | i \in I \}$.  As $\mathfrak{k}$ is the kernel of the map $e_i \mapsto a_i$, the intersection is zero if and only if $\{a_i | z_i=w_i=0\}$ is linearly independent.  The condition $z_i=w_i=0$ is the same as $\mu(z,w) \in H_i$.
\end{proof}

\begin{pro}\label{zeropro} In the above situation, $K_{(z,w)}$ is trivial if and only if $\{A_i | \mu(z,w) \in H_i \}$ is part of a $\mathbb{Z}$-basis for $X_*(H)$.
\end{pro}
\begin{proof}\
Set $I = \{ i | \mu(z,w) \in H_i \}$.  Set $I^c= \{0,1,...,n\} \setminus I$ then $\mI_{K_{(z,w)}} = (\{t^{\vec{v}}-1 | \vec{v} \in Im(\mathfrak{h}_{\mathbb{Z}}^*) \subset \mathfrak{t}_{\mathbb{Z}}^*\} \cup \{t^{e_i}| i \in I^c\})$.  Notice the set $S \subset \mathfrak{t}^*_{\mathbb{Z}}$ consisting of elements $\vec{v}$ such that $t^{d\vec{v}}-1 \in \mI_{K_{(z,w)}}$ is an additive subgroup of $X^*(T)$.  Therefore $\mI_{K_{(z,w)}} = (t_1-1,...,t_n-1)$ if and only if $S = \mathfrak{h}^* + \mathbb{Z}\{e_i | i \in I^c\} =\mathfrak{t}_{\mathbb{Z}}^*$.  The latter is true if and only if $X^*(T)=X^*(H) + \sum_{j \in I^{c}} \mathbb{Z} \chi_j^{\checkmark}$.\\

($\Rightarrow$) Suppose that $K_{(z,w)}$ is trivial, then $X^*(T) = X^*(H) + \sum_{j \in I^{c}} \mathbb{Z} \chi_j^{\checkmark}$.  For $i \in I$, write $\chi_i^{\checkmark} = \beta_i + \gamma_i$ where $\beta_i \in X^*(H)$ and $\gamma_i \in \sum_{j \in I^{c}} \mathbb{Z} \chi_j^{\checkmark}$.  $\delta_{ii'}=\langle \chi_i^{\checkmark},\chi_{i'} \rangle = \beta_i(A_{i'})$ for all $i,i' \in I$.  Therefore by basic commutative algebra (partial) delta function exist if and only if $\{A_i\}_{i \in I}$ is part of a $\mathbb{Z}$-basis for $X^*(H)$\\  

($\Leftarrow$)
Suppose $\{A_i\}_{i \in I}$ is part of a $\mathbb{Z}$-basis for $X^*(H)$.  Then there exist (partial) delta functions, $\beta_i \in X^*(H)$, such that $\beta_i(A_{i'})=\delta_{ii'}$.  Observe $0 = \delta_{ii'} - \delta_{ii'} = \langle \chi_i^{\checkmark},\chi_{i'} \rangle - \beta_i(A_{i'})$ for all $i,i' \in I$.  Therefore, $\chi_i^{\checkmark} - \beta_i \in \sum_{j \in I^{c}} \mathbb{Z} \chi_j^{\checkmark}$, implying $X^*(T) = X^*(H)+ \sum_{j \in I^{c}} \mathbb{Z} \chi_j^{\checkmark}$.\\
\end{proof}
\begin{lem}\label{alphalem}
$(z,w) \in \mu^{-1}(\lambda)^{\alpha-ss}$ if and only if $\alpha \in \sum_{z_i \neq 0} \mathbb{N}\iota^*\chi_i^{\checkmark} - \sum_{w_i \neq 0} \mathbb{N}\iota^*\chi_i^{\checkmark}$
\end{lem}
\begin{proof} $(z,w) \in \mu^{-1}(\lambda)^{\alpha-ss}$ if and only if there exists $m \in \mathbb{Z}_{>0}$ and a monomial $f = \prod_{i=1}^n x_i^{a_i} \xi_i^{b_i} \in \Gamma(T^*\mathbb{A}^n,\mO^{\alpha^m}_{T^*\mathbb{A}^n})$ such that $f(z,w) \neq 0$.\\

($\Rightarrow$) Suppose that $(z,w) \in \mu^{-1}(\lambda)$ then by the non-vanishing hypothesis on $f$, $a_i = 0$ for all $i$ such that $z_i = 0$ and $b_i=0$ for all $i$ such that $w_i = 0$.  The hypothesis that $f \in \mO^{\alpha^m}$ yields that $m \alpha = \sum_{z_i \neq 0} a_i \iota^*\chi_i^{\checkmark} - \sum_{w_i \neq 0} b_i \iota^*\chi_i^{\checkmark}$.   This implies that $\alpha \in \sum_{z_i \neq 0} \mathbb{N}\iota^*\chi_i^{\checkmark} - \sum_{w_i \neq 0} \mathbb{N}\iota^*\chi_i^{\checkmark}$.

($\Leftarrow$) Suppose that $\alpha \in \sum_{z_i \neq 0} \mathbb{N}\iota^*\chi_i^{\checkmark} - \sum_{w_i \neq 0} \mathbb{N}\iota^*\chi_i^{\checkmark}$.  Write $\alpha = \sum_{z_i \neq 0} a_i \iota^*\chi_i^{\checkmark} - \sum_{w_i \neq 0} b_i \iota^*\chi_i^{\checkmark}$. Consider the function $f = \prod_{i=1}^n x_i^{a_i} \xi_i^{b_i}$.  By construction $f(z,w) \neq 0$ and $f \in \mO^{\alpha^m}$.
\end{proof}

\begin{df}\label{walldf}
Let $I \subset \{0,1,...,n\}$ with the property $dim(\mathfrak{t}_I \cap \mathfrak{k})=1$ then define the wall of $I$, $W_I \subset \mathfrak{k}^*$, to be:

$$W_I = \{ \vec{v} | \langle \vec{v}, \mathfrak{k} \cap \mathfrak{t}_I \rangle = 0 \}$$

It clearly has dimension $k-1$.
\end{df}
Since the base field is infinite (it is algebraically closed), $\mathfrak{k} \neq \cup_{I} W_I$ as no vector space over an infinite field can be a finite union of its subspaces.  This fact, combined with the next theorem, proves that it rather common for all points in $\mu^{-1}(\lambda)^{\alpha-ss}$ to have finite stabilizers for the action of $K$.

\begin{thm}  $(d\alpha, \lambda) \notin \cup_I W_I \times W_I$ if and only if every point in $\mu^{-1}(\lambda)^{\alpha-ss}$ has a finite stabilizer for the action of $K$.
\end{thm}
\begin{proof}\
Proceed by proving the contraposition, $(d\alpha, \lambda) \in \cup_I W_I \times W_I$ if and only if there is a point in $\mu^{-1}(\lambda)^{\alpha-ss}$ which has an infinite stabilizer for the action of $K$.\\

($\Rightarrow$) Let $(z,w) \in \mu^{-1}(\lambda)^{\alpha-ss}$ and let $I = \{i | z_i=w_i=0\}$.  Suppose that $\mathfrak{t}_I \cap \mathfrak{k} \neq 0$.  By basic linear algebra there must be a set $I' \subset I$ for which $\mathfrak{t}_{I'} \cap k$ is has dimension $1$.  By \ref{alphalem},

\begin{eqnarray*}
d\alpha \in Span_{\boldk}\{\iota^*e_i^{\checkmark}|i \in I^{c} \} & \subset & \{\iota^*e_i^{\checkmark} | i \in I^{'c}\}\\
									& = & \{ \vec{v} \in \mathfrak{k}^* | \langle \vec{v},\mathfrak{t}_{I'} \rangle = 0 \}\\
									& = & W_{I'}
\end{eqnarray*}

It must be that $\lambda = \mu(z,w) = \sum_{i=0}^n z_iw_i \iota^*e_i = \sum_{i \in I^{c}}z_i w_i \iota^*e_i \in Span_{\boldk}\{\iota^*e_i | i \in I^{'c} \}$.  Implying that $(\alpha,\lambda) \in W_{I'} \times W_{I'}$.\\

($\Leftarrow$) Suppose that $(d\alpha, \lambda) \in W_I \times W_I$, and assume that $I$ is minimal (see comment below).  By the next lemma, there exists a point $(z,w) \in \mu^{-1}(\lambda)^{\alpha-ss}$ such that $I = \{i|z_i=w_i=0\}$ and $\mathfrak{k} \cap \mathfrak{t}_I \neq 0$.  Hence, $dim(K_{(z,w)}) \geq 1$.
\end{proof}

Set $Q_{\alpha,\lambda} = \{ I | \exists (z,w) \in \mu^{-1}(\lambda)^{\alpha-ss} \text{ such that } I=\{i | z_i=w_i=0\} \}$. Set $P_K = \{I | dim(\mathfrak{k} \cap \mathfrak{t}_I) = 1 \}$   It is a partially ordered set with respect to inclusion.  Notice that if $I' \subset I$ in $P_K$ then $W_I = W_{I'}$ so in what follows one might consider defining $P_K$ to contain only the minimal elements.

\begin{lem} If $I \in P_K$ is a minimal element and $d\alpha \in W_I$, then there exists $(z,w) \in \mu^{-1}(\lambda)^{\alpha-ss}$ such that $I = \{i | z_i=w_i=0\}$
\end{lem}
\begin{proof}
When $I$ is minimal, $W_I$ is spanned by $\{ \iota^*e_i | i \in I^{c} \}$.  $\alpha$ has the expansion,  $$\alpha = \sum_{i \in I^{c}}\alpha_i \iota^*\chi_i^{\checkmark} = \sum_{\substack{ i \in I^{c} \\ \alpha_i \geq 0}} \alpha_i \iota^*\chi_i^{\checkmark} + \sum_{\substack{ i \in I^{c} \\ \alpha_i < 0}} \alpha_i \iota^*\chi_i^{\checkmark}$$

This yields that $d\alpha = \sum_{\substack{ i \in I^{c} \\ \alpha_i \geq 0}} \alpha_i \iota^*e_i^{\checkmark} + \sum_{\substack{ i \in I^{c} \\ \alpha_i < 0}} \alpha_i \iota^*e_i^{\checkmark}$

Let $\tilde{\alpha}$ be some lift of $\alpha$ to $X^*(T)$.\\

Break into cases:\\
If $i \in I^{c}$ is such that $\langle \tilde{\lambda},e_i \rangle \neq 0$ then find any two numbers $z_i$ and $w_i$ such that $z_iw_i=\langle \tilde{\lambda},e_i \rangle$.\\
If $i \in I^{c}$, $\langle \tilde{\lambda},e_i \rangle = 0$ and $\langle d\tilde{\alpha},e_i \rangle \geq 0$ then choose $z_i=1$ and $w_i=0$.\\
If $i \in I^{c}$, $\langle \tilde{\lambda},e_i \rangle = 0$ and $\langle d\tilde{\alpha},e_i \rangle < 0 $ choose $z_i = 0$ and $w_i = 1$.\\
If $i \in I$, choose $z_i=w_i=0$.\\

It is clear that $I = \{i | z_i=w_i=0 \}$ and by construction, $\mu(z,w) = \lambda$ (since $\lambda \in W_I$) and $\alpha$ is in the monoid spanned over $\mathbb{N}$ by $\{ \iota^*e_i | z_i \neq 0\}$ and $\{ \iota^*e_i | w_i \neq 0 \}$.
\end{proof}

\begin{thm} $K$ acts freely on $\mu^{-1}(\lambda)^{\alpha-ss}$ if and only if $(\alpha,\lambda) \notin \cup_I W_I \times W_I$ and for all $I \in Q_{\alpha,\lambda}$ the set $\{A_i\}_{i \in I}$ is part of a $\mathbb{Z}$-basis for $X_*(H)$.
\end{thm}
\begin{proof}
($\Rightarrow$) Suppose $K$ acts freely on $\mu^{-1}(\lambda)^{\alpha-ss}$, then every point has a finite stabilizer.  By the previous theorem, this implies that $(\alpha,\lambda) \notin \cup_I W_I \times W_I$.  For the second statement, consider \ref{zeropro} yields that $K_{(z,w)} =0$ if and only if $\{a_i | \mu(z,w) \in H_i\} = \{a_i | z_i=w_i=0\}$ is part of a $\mathbb{Z}$-basis.

($\Leftarrow$) The first condition guarantees that the stabilizer of any point is finite.  The second ensures that $\{A_i | \mu(z,w) \in H_i\}$ is part of a $\mathbb{Z}$-basis for $X_*(H)$.
\end{proof}

The next corollary explains that if the collection of hyperplanes in generic enough, it is common for $(K,\alpha,\lambda)$ to be smooth.
\begin{cor} If $(\alpha,\lambda) \notin \cup_{I} W_I \times W_I$ then
 \begin{enumerate}
  \item If $\mA$ is simple, $K$ acts with finite stabilizers on $\mu^{-1}(\lambda)^{\alpha-ss}$.
  \item If $\mA$ is smooth, $K$ acts freely on $\mu^{-1}(\lambda)^{\alpha-ss}$.
 \end{enumerate}

\end{cor}

\section{Acknowledgments}

The author benefited from conversations with Sam Gunningham and Michael Skirvin.  Correspondences with Ben Webster assisted with the combinatorial condition found in \ref{weakconj}.  He also provided helpful comments on an earlier version of this paper.  The author is indebted to David Treumann, most of all for his outstanding mentorship, but also for the reformulation contained in \ref{cart} and his suggestion to expand the concepts of \cite{MVdB}, \cite{BeK}, and \cite{PW} to positive characteristic.  The author would also like to extend deep thanks to his advisor David Nadler for years of continual support, assistance and encouragement, with a special gratitude for cultivating the author's interest in geometric representation theory.  This work grew from the pursuit of \ref{smainthm}, which is a question that he posed to the author.

Department of Mathematics, Northwestern University, 2033 Sheridan Road, Evanston, IL 60202\\
\textit{E-mail: tstadnik@math.northwestern.edu}\\
\end{document}